\documentclass[11pt,psamsfonts]{amsart}
\usepackage{amsmath}
\usepackage{amsthm}
\usepackage{amssymb}
\usepackage{amscd}
\usepackage{amsfonts}
\usepackage{amsbsy}
\usepackage{graphicx}
\usepackage[dvips]{psfrag}
\usepackage{array}
\usepackage{color}
\usepackage{epsfig}
\usepackage{url}

\newcolumntype{L}{>{\displaystyle}l}
\newcolumntype{C}{>{\displaystyle}c}
\newcolumntype{R}{>{\displaystyle}r}

\newcommand{\R}{\ensuremath{\mathbb{R}}}

\newcommand{\CC}{\mathcal{C}}

\newcommand{\CO}{\ensuremath{\mathcal{O}}}
\newcommand{\ov}{\overline}

\newcommand{\T}{\theta}

\newcommand{\sgn}{\mathrm{sign}}

\def\e{\varepsilon}

\newtheorem {theorem} {Theorem}

\newtheorem {lemma} [theorem] {Lemma}

\newtheorem {remark}[theorem] {Remark}

\begin{document}

\title[Limit cycles of control piecewise linear differential systems]
{Limit cycles for two classes of control piecewise linear
differential systems}

\author[J. Llibre, R.D. Oliveira and C.A.B. Rodrigues]
{Jaume Llibre$^1$, Regilene D. Oliveira$^{2}$  and Camila Ap. B.
Rodrigues$^3$}

\address{$^1$ Departament de Matematiques,
Universitat Aut\`{o}noma de Barcelona, 08193 Bellaterra, Barcelona,
Catalonia, Spain} \email{jllibre@mat.uab.cat}

\address{$^{2,3}$  Department of Mathematic, ICMC-Universidade de S\~{a}o Paulo,
Avenida Trabalhador S\~ao-carlense, 400, 13566-590, S\~{a}o Carlos, SP, Brazil}
\email{regilene@icmc.usp.br, camilaap@icmc.usp.br}

\begin{abstract}
We study the bifurcation of limit cycles from the periodic orbits of
$2n$--dimensional linear centers $\dot{x} = A_0 x$ when they are
perturbed inside classes of continuous and discontinuous piecewise
linear differential systems of control theory of the form $\dot{x} =
A_0 x + \e \big(A x + \phi(x_1) b\big)$, where $\phi$ is a
continuous or discontinuous piecewise linear function, $A_0$ is a
$2n\times 2n$ matrix with only purely imaginary eigenvalues, $\e$ is
a small parameter, $A$ is an arbitrary $2n\times 2n$ matrix, and $b$
is an arbitrary vector of $\R^n$.
\end{abstract}



\maketitle

\section{Introduction and statement of the main results}

In control theory are relevant the {\it continuous piecewise linear
differential systems} of the form
\begin{equation}\label{E1}
\dot x= A x + \varphi(x_1) b,
\end{equation}
with $A$ a $m\times m$ matrix, $x,b\in \R^m$, $\varphi: \R \to \R$
is the continuous piecewise linear function
\begin{equation}\label{eq phi}
\varphi(x_1) = \left\{
                        \begin{array}{RL}
                          -1 & \hbox{if $x_1 \in (-\infty,-1)$,} \\
                          x_1 & \hbox{if $x_1 \in [-1,1]$,} \\
                          1 & \hbox{if $x_1 \in (1,\infty)$,}
                        \end{array}
                      \right.
\end{equation}
where $x = (x_1,\ldots,x_m)^T$, and the dot denotes the derivative
with respect to the independent variable $t$, the time.

Also in control theory are important the {\it discontinuous
piecewise linear differential systems} of the form \eqref{E1} where
instead of the function $\varphi$ we have the discontinuous
piecewise linear function
\begin{equation}\label{eq phi 0}
\psi(x_1) = \left\{
                        \begin{array}{RL}
                          -1 & \hbox{if $x_1 \in (-\infty,0)$,} \\
                           1 & \hbox{if $x_1 \in (0,\infty)$.}
                        \end{array}
                      \right.
\end{equation}
For more details on these continuous and discontinuous piecewise
linear differential systems see for instance the books \cite{Ai, BC,
Gi, Kh, Me, Vi}.

The analysis of discontinuous piecewise linear differential systems
goes back mainly to Andronov and coworkers \cite{AVK-1966} and
nowadays still continues to receive attention by many researchers.
In particular, discontinuous piecewise linear differential systems
appear in a natural way in control theory and in the study of
mechanical systems, electrical circuits, ... see for instance the
book \cite{BBCK-2008} and the references quoted there. These systems
can present complicated dynamical phenomena such as those exhibited
by general nonlinear differential systems.

One of the main ingredients in the qualitative description of the
dynamical behavior of a differential system is the number and the
distribution of its limit cycles. The goal of this paper is to study
analytically the existence of limit cycles for a class of continuous
and a class of discontinuous piecewise linear differential of the
form \eqref{E1}.

More precisely, first we consider the class of continuous piecewise
linear differential systems
\begin{equation}\label{eq inicial}
\dot{x} = A_0 x  + \e \big(A x + \varphi(x_1) b\big),
\end{equation}
with $| \e | \neq 0$ a sufficiently small real parameter, where
$A_0$ is the $2n\times 2n$ matrix having on its principal diagonal
the following $2\times 2$ matrices
$$
\left(
\begin{array}{cc}
 0 & -(2k-1)\\
2k-1 & 0
\end{array}
\right) \qquad \mbox{for $k=1,\ldots,n$,}
$$
and zeros in the complement, $A$ is an arbitrary $2n\times 2n$
matrix and $b\in \R^{2n} \setminus \{0 \}$. Note that for $\e = 0$
system \eqref{eq inicial} becomes
\begin{equation}\label{eq epsilon zero}
\dot{x}_1 = -x_2, \quad \dot{x}_2 = x_1, \quad \ldots \quad
,\dot{x}_{2n-1} = -(2n-1) x_{2n}, \quad \dot{x}_{2n} = (2n-1) x_{2n-1}.
\end{equation}
Moreover, the origin of \eqref{eq epsilon zero} is a \textit{global
isochronous center} in $\R^{2n}$, i.e. all its orbits different from
the origin are periodic with period $2 \pi$.

In a similar way we consider the discontinuous piecewise linear
differential systems
\begin{equation}\label{eq iniciald}
\dot{x} = A_0 x  + \e \big(A x + \psi(x_1) b\big).
\end{equation}

Our main results on the limit cycles of the continuous and
discontinuous piecewise linear differential systems \eqref{eq
inicial} are the following ones.

\begin{theorem}\label{main theorem}
For $|\e|>0$ sufficiently small and if the conditions for applying
the averaging theory of first order hold, then at most one limit
cycle $\gamma_\e$ of the continuous piecewise linear differential
system \eqref{eq inicial} bifurcates from the periodic orbits of
system \eqref{eq epsilon zero}, i.e. $\gamma_\e$ tends to a periodic
solution of system \eqref{eq epsilon zero} when $\e\to 0$. Moreover
there are systems \eqref{eq inicial} with $|\e|>0$ sufficiently
small having a such limit cycle.
\end{theorem}

Theorem \ref{main theorem} is proved in section \ref{s3}.

\begin{theorem}\label{main theoremd}
For $|\e|>0$ sufficiently small and if the conditions for applying
the averaging theory of first order hold, then at most one limit
cycle $\gamma_\e$ of the discontinuous piecewise linear differential
system \eqref{eq iniciald} bifurcates from the periodic orbits of
system \eqref{eq epsilon zero}. Moreover there are systems \eqref{eq
iniciald} with $|\e|>0$ sufficiently small having a such limit
cycle.
\end{theorem}

Theorem \ref{main theoremd} is proved in section \ref{s4}.

If instead of the matrix $A_0$ we consider the matrix $A_1$ where
$A_1$ is the $2n\times 2n$ matrix having on its principal diagonal
the following $2\times 2$ matrices
$$
\left(
\begin{array}{cc}
 0 & -k\\
k & 0
\end{array}
\right) \qquad \mbox{for $k=1,\ldots,n$,}
$$
and zeros in the complement, then we have the following results.

\begin{theorem}\label{main theorem1}
Assume that the conditions for applying the averaging theory of
first order hold. Then this theory does not provide any information
about the limit cycles of the continuous piecewise linear
differential system
\begin{equation}\label{eq inicial1}
\dot{x} = A_1 x  + \e \big(A x + \varphi(x_1) b\big).
\end{equation}
\end{theorem}

\begin{theorem}\label{main theoremd1}
Assume that the conditions for applying the averaging theory of
first order hold. Then this theory does not provide any information
about the limit cycles of the discontinuous piecewise linear
differential system
\begin{equation}\label{eq inicial1d}
\dot{x} = A_1 x  + \e \big(A x + \psi(x_1) b\big).
\end{equation}

\end{theorem}

Theorems \ref{main theorem1} and \ref{main theoremd1} are proved in
section \ref{s5}.

Note the difference between the matrices $A_0$ and $A_1$, in the
matrix $A_0$ the non-zero entries are only the odd numbers
$1,3,\ldots,2n-1$, while in the matrix $A_1$ the non-zero entries
are the numbers $1,2,\ldots,n$. This difference provides that the
continuous and discontinuous piecewise linear differential systems
\eqref{eq inicial} and \eqref{eq iniciald} can have limit cycles
detected by the averaging theory, while for the continuous and
discontinuous piecewise linear differential systems \eqref{eq
inicial1} and \eqref{eq inicial1d} the averaging theory cannot
detect limit cycles.

According to the results of the averaging theory used it follows
that for the control differential systems here studied, the limit
cycles that we obtain bifurcate from some periodic orbit of the
$2n$-dimensional linear differential center \eqref{eq epsilon zero}.
This technique of finding limit cycles bifurcating from centers has
been intensively studied in dimension $2$, see for instance the book
of Christopher and Li \cite{CL} and the hundreds of references
quoted therein.

Other results different to the ones presented here, but which also
study the limit cycles of control systems of the form \eqref{E1}
using averaging theory, can be found in \cite{BC2, BLMT, CCL, LM, discrete}.

The main tools for proving the previous theorems are the extensions
of the classical averaging theory for computing periodic solutions
of $\mathcal C^2$ differential systems to continuous and
discontinuous differential systems. In section \ref{s2} we summarize
the extensions of the averaging theory that we shall use here for
proving our results.

\section{First order averaging theory}\label{s2}

For the classical averaging theory for finding periodic orbits of
differential systems of class $\mathcal C^2$ see for instance the
chapter 11 of the book of Verhulst \cite{Ve}.

In this section we present first the result on the continuous
averaging theory that we will use for proving our Theorems \ref{main
theorem} and \ref{main theorem1}. This theory uses the Brouwer
degree of a continuous function and its proof can be found in
\cite{Buica-Llibre-2004}.

\begin{theorem}\label{averaging theorem}
We consider the following differential system
\begin{equation}\label{eq inicial averaging}
\dot x = \e H(t,x) + \e^2R(t,x,\e),
\end{equation}
where $H:\R \times D \to \R^n$, $R:\R \times D \times (-\e_f,\e_f)
\to \R^n$ are continuous functions, $T$-periodic in the first
variable, and $D$ is an open bounded subset of $\R^n$. We define
$h:D \to \R^n$ as
\begin{equation}\label{averaged function theorem}
h(z)=\int_0^T H(s,z)ds,
\end{equation}
and assume that
\begin{itemize}
\item[(i)] $H$ and $R$ are locally Lipschitz with respect to $x$;

\item[(ii)] for $p \in D$ with $h(p)=0$, there exists a
neighborhood $V$ of $p$ such that $h(z) \neq 0$ for all $z \in
\bar{V} \setminus \{p\}$ and the Brouwer degree $d_B(h,V,0) \neq 0$.
\end{itemize}
Then, for $|\e|\neq 0$ sufficiently small, there exists an isolated
$T$- periodic solution $x(t,\e)$ of system \eqref{eq inicial
averaging}such that $x(0,\e) \to p$ as $\e \to 0$.
\end{theorem}

\begin{remark}\label{r1}
Let $h:D \to \R^n$ be a $C^1$ function with $h(p)=0$, where $D$ is
an open bounded subset of $\R^n$ and $p \in D$. If the Jacobian of
$h$ at $p$ is not zero, then there exists a neighborhood $V$ of $p$
such that $h(z) \neq 0$ for all $z \in \bar{V} \setminus \{p\}$, and
the Brouwer degree $d_B(h,V,p) \in \{-1,1\}$.
\end{remark}

For a proof of Remark \ref{r1} see for instance \cite{Ll}.

For proving Theorems \ref{main theoremd} and \ref{main theoremd1} we
shall need the following extension of the averaging theory for
computing periodic solutions to discontinuous differential systems
done in \cite{LNT}.

\begin{theorem}\label{discontinuous}
We consider the following discontinuous differential system
\begin{equation}\label{MRs1}
x'(t)=\e H(t,x)+\e^2R(t,x,\e),
\end{equation}
with
\[
\begin{aligned}
&H(t,x)=H_1(t,x)+\sgn(g(t,x))H_2(t,x),\\
&R(t,x,\e)=R_1(t,x,\e)+\sgn(g(t,x))R_2(t,x,\e),
\end{aligned}
\]
where $H_1,H_2:\R\times D\rightarrow\R^n$, $R_1,R_2:\R\times
D\times(-\e_0,\e_0)\rightarrow\R^n$ and $g:\R\times D\rightarrow \R$
are continuous functions, $T$--periodic in the variable $t$ and $D$
is an open subset of $\R^n$. We also suppose that $g$ is a $\CC^1$
function having $0$ as a regular value.

Define the average function $h:D\rightarrow\R^n$ as
\begin{equation}\label{MRf1}
h(x)=\int_0^T H(t,x) dt.
\end{equation}
We assume the following conditions.
\begin{itemize}
\item[$(i)$] $H_1,\,H_2,\,R_1,\,R_2$ are locally Lipschitz
with respect to $x$;

\item[$(ii)$] there exists an open bounded subset $C\subset D$ such
that, for $|\e|>0$ sufficiently small, every orbit starting in $\ov
C$ reaches the set of discontinuity only at its crossing regions.

\item[$(iii)$] for $a\in C$ with $h(a)=0$, there exists a neighbourhood
$U\subset C$ of $a$ such that $h(z)\neq 0$ for all $z\in\overline{U}
\setminus \{a\}$ and $d_B(h,U,0)\neq 0$.
\end{itemize}
Then, for $|\e|>0$ sufficiently small, there exists a $T$--periodic
solution $x(t,\e)$ of system \eqref{MRs1} such that $x(0,\e)\to a$
as $\e\to 0$.
\end{theorem}

\section{Proof of Theorem \ref{main theorem}}\label{s3}

The main tool for proving Theorem \ref{main theorem} is the
averaging theory of first order for continuous differential systems
presented in Theorem \ref{averaging theorem}. In order to use this
theorem we need to write the differential system \eqref{eq inicial}
in the normal form \eqref{eq inicial averaging}, and for obtaining
this we need to some changes of variables.

\begin{lemma}\label{polarsystem}
Doing the change of variables $(x_1,x_2,\ldots,x_{2n}) \mapsto
(\T,r,\T_1,r_1,\ldots,\T_{n-1},r_{n-1})$ defined by
\begin{equation*}
\begin{split}
&x_1=r\cos\T, \\
&x_2=r\sin\T,\\
&x_{2j-1}=r_{j-1}\cos((2j-1)\T+\T_{j-1}), \\
&x_{2j}=r_{j-1} \sin((2j-1)\T+\T_{j-1}),
\end{split}
\end{equation*}
for $j=2,\ldots,n$ system \eqref{eq inicial} is transformed into the
system
\begin{equation}\label{polarvariables}
\begin{array}{RL}
\frac{dr}{d\T}=&\e H_1(\T,r,\T_1,r_1,\ldots,\T_{n-1},r_{n-1}) + \CO(\e^2),\\
\frac{dr_{j-1}}{d\T}=&\e
H_{2(j-1)}(\T,r,\T_1,r_1,\ldots,\T_{n-1},r_{n-1}) + \CO(\e^2),\\
\frac{d\T_{j-1}}{d\T}=&\e
H_{2j-1}(\T,r,\T_1,r_1,\ldots,\T_{n-1},r_{n-1}) + \CO(\e^2),
\end{array}
\end{equation}
where
\begin{equation*}
\begin{array}{RL}
H_1=&\sum_{l=1}^n r_{l-1} \bigg(F_{1,l}\cos \T+F_{2,l}\sin
\T\bigg)+\varphi(r \cos \T)(b_1 \cos \T + b_2 \sin \T), \\
\end{array}
\end{equation*}
and for $j=2,3,\ldots,n$ we have
\begin{equation*}
\begin{array}{RL}
H_{2(j-1)}=&\sum_{l=1}^n r_{l-1} \bigg(F_{2j-1,l} \cos
((2j-1)\T+\T_{j-1}) + F_{2j,l}\sin ((2j-1)\T+\T_{j-1}) \bigg)\\
&+\varphi(r \cos \T)\big[b_{2j-1} \cos ((2j-1)\T+\T_{j-1}) + b_{2j}
\sin ((2j-1)\T+\T_{j-1})\big],\\
\end{array}
\end{equation*}
\begin{equation*}
\begin{array}{RL}
H_{2j-1}=& \sum_{l=1}^n \frac{r_{l-1}}{r_{j-1}} \bigg( F_{2j,l}\cos
((2j-1)\T+\T_{j-1}) -F_{2j-1,l}\sin ((2j-1)\T+\T_{j-1})\bigg)\\
&+(2j-1)\sum_{l=1}^n \frac{r_{l-1}}{r} \bigg(F_{1,l}\sin
\T-F_{2,l}\cos \T \bigg)\\
&+\varphi(r \cos \T)\bigg(\frac{b_{2j}}{r_{j-1}} \cos
((2j-1)\T+\T_{j-1}) - \frac{b_{2j-1}}{r_{j-1}} \sin
((2j-1)\T+\T_{j-1})\bigg)\\
&-(2j-1)\varphi(r \cos \T) \bigg(\frac{b_2}{r}\cos \T -
\frac{b_1}{r} \sin \T \bigg),
\end{array}
\end{equation*}
with
\begin{equation*}
F_{i,l}=F_{i,l}(r,\T,\T_{l-1})=a_{i(2l-1)} \cos ((2l-1)\T +
\T_{l-1}) + a_{i(2l)} \sin ((2l-1)\T+\T_{l-1}).
\end{equation*}
We take $\e_0$ sufficiently small, $m$ arbitrarily large and
\[
D_m= \bigg\{(r,\T_1,r_1,\ldots,\T_{n-1},r_{n-1}) \in
\bigg(\frac{1}{m},m\bigg) \times \bigg[\mathbb S^1 \times
\bigg(\frac{1}{m},m\bigg)\bigg]^{n-1}\bigg\}.
\]
Then the vector field of system \eqref{polarvariables} is well
defined and continuous on $\mathbb S^1 \times D_m \times
(-\e_0,\e_0)$. Moreover the system is $2\pi$-periodic with respect
to variable $\T$ and locally Lipschitz with respect to variables
$(r,\T_1,r_1,\ldots,\T_{n-1},r_{n-1})$.
\end{lemma}

\begin{proof}
In the variables $(\T,r,\T_1,r_1,\ldots,\T_{n-1},r_{n-1})$ the
differential system \eqref{eq inicial} becomes
\begin{equation*}
\begin{array}{RL}
\dot{\T}=&1+\dfrac{\e}{r}\bigg[\sum_{l=1}^n r_{l-1}
\bigg(F_{2,l}\cos \T -F_{1,l} \sin \T\bigg)+\varphi(r \cos
\T)(b_2\cos \T - b_1 \sin
\T)\bigg],  \\
\dot{r}=&\e H_1(\T,r,\T_1,r_1,\ldots,\T_{n-1},r_{n-1}),\\
\dot{r}_{j-1}=&\e  H_{2(j-1)}(\T,r,\T_1,r_1,\ldots,\T_{n-1},r_{n-1}),\\
\dot{\T}_{j-1}=&\e  H_{2j-1}(\T,r,\T_1,r_1,\ldots,\T_{n-1},r_{n-1}),
\end{array}
\end{equation*}
for $j=2,3,\ldots,n$. Note that for $\e = 0$ , $\dot \T (t)>0$ and
hence for $|\e| \neq 0$ sufficiently small this property remains
valid for each $t$ when $(\T,r,\T_1,r_1,\ldots,\T_{n-1},r_{n-1}) \in
\mathbb{S}^1 \times D_m$. Now we take $\T$ as the new independent
variable. The right-hand side of the new system is well defined and
continuous in $\mathbb S^1 \times D_m \times (-\e_0,\e_0)$ and it is
$2\pi$-periodic with respect to the new variable $\T$ and locally
Lipschitz with respect to $(r,\T_1,r_1,\ldots,\T_{n-1},r_{n-1})$.
Now system \eqref{polarsystem} can be obtained doing a Taylor series
expansion in the parameter $\e$ around $\e=0$.
\end{proof}

The next step is to find the corresponding average function
\eqref{averaged function theorem} of system \eqref{polarsystem} that
we denoted by $h=(h_1,h_2,\ldots,h_{2(n-1)},h_{2n-1}):D_m \to
\R^{n-1}$ and it is defined by
\begin{equation*}
\begin{array}{RL}
h_1=&h_1(r,\T_1,r_1,\ldots,\T_{n-1},r_{n-1})=\int_{0}^{2\pi}
H_1(r,\T_1,r_1,\ldots,\T_{n-1},r_{n-1}) d\T,\\
h_{2(j-1)}=&h_{2(j-1)}(r,\T_1,r_1,\ldots,\T_{n-1},r_{n-1})=\int_{0}^{2\pi}
H_{2(j-1)}(r,\T_1,r_1,\ldots,\T_{n-1},r_{n-1}) d\T,\\
h_{2j-1}=&h_{2j-1}(r,\T_1,r_1,\ldots,\T_{n-1},r_{n-1})=\int_{0}^{2\pi}
H_{2j-1}(r,\T_1,r_1,\ldots,\T_{n-1},r_{n-1}) d\T,
\end{array}
\end{equation*}
for $j=1,2,\ldots,n$. To calculate these integrals we will use the
following equalities
\begin{equation*}
\begin{array}{RL}
&\int_0^{2\pi} \cos((2j-1)\T+\T_{j-1})\sin((2l-1)\T+\T_{l-1})d\T=0
\qquad \text{for all integers $l,j>1$},\\ &\int_0^{2\pi} \cos((2j-1)\T+\T_{j-1})
\cos((2l-1)\T+\T_{l-1})d\T =\begin{cases}
\pi & \text{if $l=j$,}\\
0 & \text{if $l\neq j$,}
\end{cases}\\
\end{array}
\end{equation*}
\begin{equation*}
\begin{array}{RL}
&\int_0^{2\pi} \sin((2j-1)\T+\T_{j-1})\sin((2l-1)\T+\T_{l-1})d\T =
\begin{cases}
\pi & \text{if $l=j$,}\\
0 & \text{if $l\neq j$,}
\end{cases}\\
\end{array}
\end{equation*}
and the next lemma.

For $r>0$ and $j=1,2,\ldots,n$ we denote
\begin{equation*}
\begin{array}{RL}
I_j(r)=&\int_0^{2\pi} \varphi(r \cos \T) \cos((2j-1)\T) d\T,\\
J_j(r)=&\int_0^{2\pi} \varphi(r \cos \T) \sin((2j-1)\T) d\T,\\
\end{array}
\end{equation*}
where $\varphi$ is the piecewise linear function \eqref{eq phi}.

\begin{lemma}\label{lemma4}
The integrals $I_j$ and $J_j(r)$ satisfy
\begin{equation*}
\begin{array}{RL}
I_j(r)=&\begin{cases}
\pi r & \text{if $j=1$ and $0<r\leq 1$},\\
0 & \text{if $j>1$ and $0<r\leq 1$},\\
K(r) & \text{if $j=1$ and $r>1$},\\
L_j(r)& \text{if $j>1$ and $r>1$};
\end{cases}\\
J_j(r)=&0 \quad \quad \text{for all $j=1,2,\ldots,n$ and $r>0$.}
\end{array}
\end{equation*}
where
\begin{equation*}
\begin{array}{RL}
L_j(r)=&\frac{2}{j(2j-1)^2}\bigg((2j-1) \sqrt{-1+r^2} \cos ((2j-1)
\arctan \sqrt{-1+r^2})\\
&- \sin ((2j-1) \arctan \sqrt{-1+r^2})\bigg),\\
K(r)=&\pi r + \dfrac{2}{r} \sqrt{r^2-1}-2r \arctan(\sqrt{r^2-1}).
\end{array}
\end{equation*}
\end{lemma}

\begin{proof}
We consider two cases: $0<r\leq 1$ and $r>1$.

\noindent\textbf{Case 1: $0<r\leq 1$} In this case $|r \cos \T| \leq
1 $ and hence $\varphi(r \cos \T)=r \cos \T$ for all $\T \in
[0,2\pi]$. Then if $j=1$
\[
\int_0^{2\pi} \varphi(r \cos \T) \cos\T d\T=r\int_0^{2\pi} \cos^2 \T
d\T=\pi r,
\] and
\[
\int_0^{2\pi} \varphi(r \cos \T) \sin\T d\T=r\int_0^{2\pi} \cos \T
\sin \T d\T=0.
\] And if $j>1$ then
\[
\int_0^{2\pi} \varphi(r \cos \T) \cos((2j-1)\T) d\T=r\int_0^{2\pi}
\cos \T\cos((2j-1)\T) d\T=0,
\]

\[
\int_0^{2\pi} \varphi(r \cos \T) \sin((2j-1)\T) d\T=r\int_0^{2\pi}
\cos \T\sin((2j-1)\T) d\T=0.
\]

\noindent\textbf{Case 2: $r>1$} In this case choose $\T_c \in
(0,\pi/2)$ such that $\cos \T_c=1/r$. If $j=1$ we have
\begin{equation*}
\begin{array}{RL}
I_1(r)=&\int_{0}^{\T_c} \cos \T  d\T + r \int_{\T_c}^{\pi-\T_c}
\cos^2 \T  d\T - \int_{\pi-\T_c}^{\pi+\T_c} \cos \T  d\T\\
&+ r \int_{\pi+\T_c}^{2\pi-\T_c} \cos^2 \T  d\T +
\int_{2\pi-\T_c}^{2\pi} \cos \T  d\T\\
=&\pi r + \frac{2}{r}
\sqrt{r^2-1}-2r \arctan(\sqrt{r^2-1}).
\end{array}
\end{equation*}
The same reasoning can be applied to see that $J_1(r)=0$. If $j>1$ then
\begin{equation*}
\begin{array}{RL}
I_j(r)=&\int_{0}^{\T_c} \cos ((2j-1)\T)  d\T + r
\int_{\T_c}^{\pi-\T_c} \cos \T  \cos((2j-1)\T) d\T -
\int_{\pi-\T_c}^{\pi+\T_c} \cos((2j-1)\T)  d\T\\
&+ r \int_{\pi+\T_c}^{2\pi-\T_c} \cos \T  \cos ((2j-1)\T) d\T +
\int_{2\pi-\T_c}^{2\pi} \cos ((2j-1)\T)  d\T\\
=&\frac{2}{j(2j-1)^2}\bigg((2j-1) \sqrt{-1+r^2} \cos ((2j-1) \arctan
\sqrt{-1+r^2})\\ &- \sin ((2j-1) \arctan \sqrt{-1+r^2})\bigg),
\end{array}
\end{equation*} and $J_j(r)=0$.
\end{proof}

With the results presented previously we are able to prove Theorem
\ref{main theorem}. Since we can choose $m$ sufficiently large to
find the zeroes of the average function $h$ in $D_m$ it is
sufficient to look for them in $(0,\infty) \times [\mathbb S^1
\times (0,\infty)]^{n-1}$. To calculate the expression of the
average function we consider again two cases.

\noindent{\bf Case} 1: $0<r\leq 1$. In this case the system whose
zeros can provide limit cycles of system \eqref{eq inicial} is
\begin{equation}\label{jiji}
\begin{array}{RL}
h_1=&(a_{11} + a_{22}+b_1)\pi r,\\
h_2=&(a_{33} + a_{44})\pi r_{1},\\
h_3=&(a_{43} - a_{34} +3(a_{12}-a_{21}-b_2))\pi,\\
&\vdots\\
h_{2(n-1)}=&(a_{(2n-1)(2n-1)} + a_{(2n)(2n)}))\pi r_{n-1},\\
h_{2n-1}=&(a_{(2n)(2n-1)} - a_{(2n-1)(2n)}
+(2n-1)(a_{12}-a_{21}-b_2)\pi.
\end{array}
\end{equation}
Note that the variables $\T_1,\T_2,\ldots,\T_{n-1}$ does not appear
explicitly into system \eqref{jiji}. Hence, if this system has
zeros, it has a continuum of zeros. Therefore the assumption (ii) of
the averaging theory, presented in Theorem \ref{averaging theorem},
is not satisfied and this theorem does not provide any information
about the limit cycles of system \eqref{polarvariables}.

\noindent{\bf Case} 2: $r>1$. Now the system whose zeros can provide
limit cycles of system \eqref{polarvariables} is
\begin{equation}\label{case r>1}
\begin{array}{RL}
h_1=&(a_{11} + a_{22})\pi r + b_1 K(r),\\
h_{2}=&(a_{33} + a_{44})\pi r_{1} + (b_3 \cos \T_{1} + b_{4}\sin \T_{2}) L_2(r),\\
h_{3}=&(a_{43} - a_{34} +3(a_{12}-a_{21}))\pi -\\
&\frac{3 b_2 r_{1}K(r) -r (b_{4}\cos \T_{1} - b_{3}\sin \T_{1})
L_{2}(r)}
{r r_1},\\
&\vdots\\
h_{2(n-1)}=&(a_{(2n-1)(2n-1)} + a_{(2n)(2n)})\pi r_{n-1} + \\
& (b_{2n-1} \cos \T_{n-1} + b_{2n}\sin \T_{n-1}) L_{n}(r),\\
h_{2n-1}=&(a_{(2n)(2n-1)} - a_{(2n-1)(2n)} +(2n-1)(a_{12}-a_{21}))\pi -\\
&\frac{(2n-1) b_2 r_{n-1}K(r) -r (b_{2n}\cos \T_{n-1} - b_{2n-1}\sin
\T_{n-1}) L_{n}(r)}{r r_{n-1}},
\end{array}
\end{equation}
For each $j\in \{2,3,\ldots,n\}$ we will study the zeros of the
system
\begin{equation*}\label{arbitrary system}
\begin{array}{RL}
h_1=&(a_{11} + a_{22})\pi r + b_1 K(r),\\
h_{2(j-1)}=&(a_{(2j-1)(2j-1)} + a_{(2j)(2j)})\pi r_{j-1} +\\
& (b_{2j-1}\cos \T_{j-1} + b_{2j}\sin \T_{j-1}) L_{j}(r),\\
h_{2j-1}=&(a_{(2j)(2j-1)} - a_{(2j-1)(2j)} +(2j-1)(a_{12}-a_{21}))\pi- \\
&\frac{(2j-1) b_2 r_{j-1}K(r) -r (b_{2j}\cos \T_{j-1} - b_{2j-1}\sin
\T_{j-1}) L_{j}(r)}{r r_{j-1}},
\end{array}
\end{equation*}

\noindent{\bf Claim}: {\it The function $K:(1,\infty) \to (\pi,4)$
is a diffeomorphism}. Indeed note that $K$ is twice differentiable
with
\[
K'(r)=\pi-2\frac{\sqrt{r^2-1}}{r^2}-2 \arctan \sqrt{r^2-1},
\]
and
\[
K''(r)=-\frac{4}{r^3 \sqrt{r^2-1}}<0
\]
which implies that $K'$ is a strictly decreasing function. Moreover
$\lim_{r\to \infty} K'(r)=0$ what means that $K'(r)$ has a
horizontal asymptote given by the axis $r$ and then $K'(r) \geq 0$.
Suppose that there exists an $r_0 \in (1,\infty)$ such that
$K'(r_0)=0$. Then for all $r>r_0$ we have $K'(r)<K'(r_0)=0$,
contradiction. Therefore it follows that $K'(r) \neq 0$ for all $r
\in (1,\infty)$ and the Inverse Function Theorem guarantees that $K$
is a local diffeomorphism and since that $K$ is a injective function
we obtain the global diffeomorphism, ending the proof of this claim.

First we note that in order that the equation $h_1=0$ has solutions
with $r>1$ it is necessary that $b_1(a_{11}+a_{22})<0$. Moreover
$K''(r)<0$ implies that the graph of $K$ is convex. In the plane of
the graph of $K(r)$ the graph of $(a_{11}+a_{22})\pi r$ is a
straight line passing through the origin and then both graphs can
intersect at most in two points.

\begin{figure}[htbp]
\centerline{\includegraphics [width=5cm]{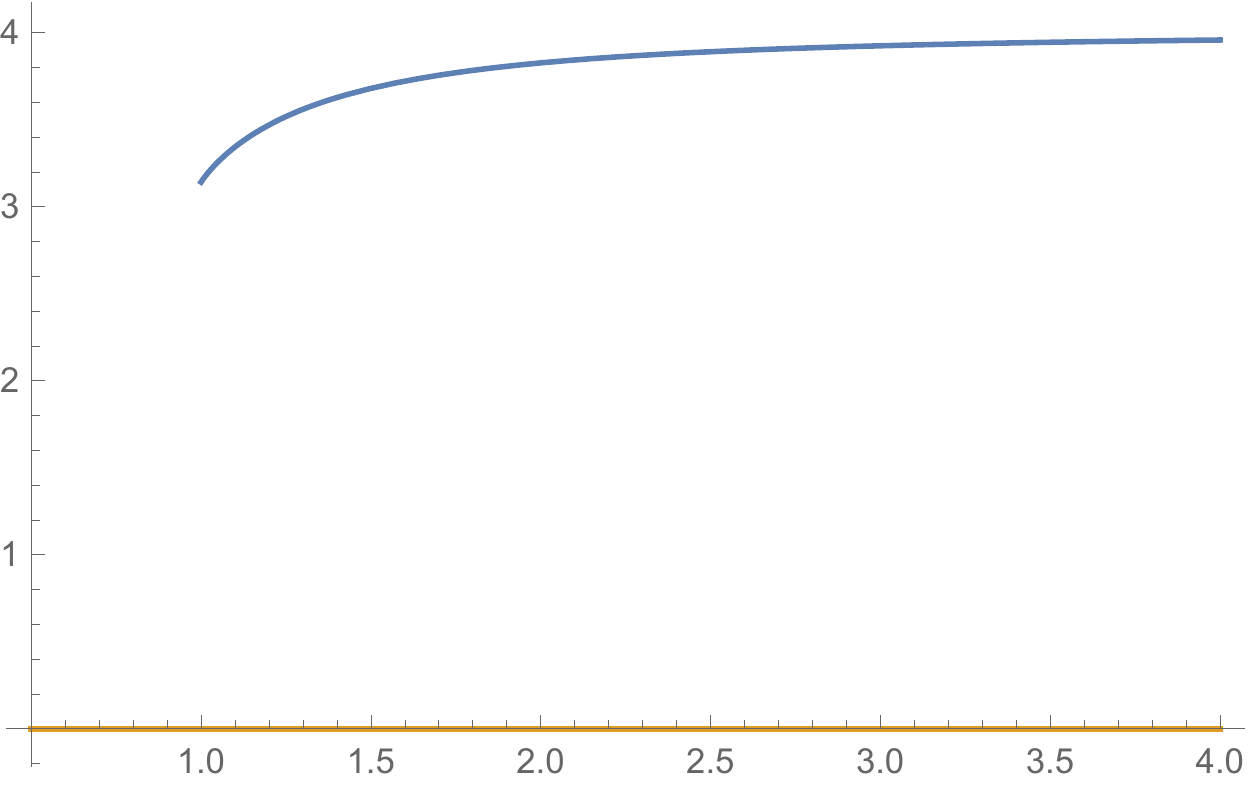}}\centerline {}
\caption{The graphic of the function $K(r)$.}
\label{figure}\centerline{}
\end{figure}

But if some straight line intercept the graph of $K(r)$ in two
points then it cannot pass through the origin, as we can see in
Figure \ref{figure}. Then the equation $h_1=0$ has at most one
solution if $r>1$, and since that $K(r)$ is a diffeomorphism we can
choose the coefficients $a_{11}, a_{22}$ and $b_1$ so that this
solution exists. We denote this solution by $r_0$ and we substitute
it into the equations $h_{2(j-1)}=0$ and $h_{2j-1}=0$. Defining
\begin{equation*}
\begin{array}{RL}
A_j=&(a_{(2j-1)(2j-1)} + a_{(2j)(2j)})\pi, \quad \quad B_j=b_{2j-1}
L_{j}(r_0), \quad \quad C_j=b_{2j} L_{j}(r_0),\\
D_j=& (a_{(2j)(2j-1)} - a_{(2j-1)(2j)}
+(2j-1)(a_{12}-a_{21}))\pi-\frac{1}{r_{0}}(2j-1) b_2 K(r_0),\\
u_j=&\cos \T_{j-1}, \quad \quad v_j=\sin \T_{j-1},
\end{array}
\end{equation*}
the system $h_{2(j-1)}=h_{2j-1}=0$ is equivalent to the system
\begin{equation*}
\begin{array}{RL}
A_j r_{j-1} + B_j u_j + C_j v_j = & 0,\\
D_j r_{j-1} + C_j u_j - B_j v_j= & 0,\\
u_j^2+v_j^2-1=&0.
\end{array}
\end{equation*}
Using the two first equations we obtain
\[
u_j=-\frac{(A_j B_j + C_j D_j) r_{j-1}}{B_j^2+C_j^2},\quad \quad
v_j=\frac{(B_j D_j -A_j C_j) r_{j-1}}{B_j^2+C_j^2}.
\]
Substituting these two expressions in the third equation we get
\[
(A_j^2+D_j^2)r_{j-1}^2-B_j^2-C_j^2=0.
\]
Therefore at most there is one solution $r_{j-1}>0$, which provide a
unique $u_j$ and $v_j$. Since we fixed an arbitrarily $j$ to solve
this system, the same reasoning can be applied to each pair of
equations $h_{2(j-1)}=0$ and $h_{2j-1}=0$, concluding that system
\eqref{case r>1} has at most one solution. Moreover taking
conveniently the parameters of the initial system \eqref{eq inicial}
this solution exists and its Jacobian is not zero. Hence at most one
limit cycle can bifurcate from the periodic orbits of the center of
system \eqref{eq epsilon zero} when we perturbe it as in system
\eqref{eq inicial}, and there are systems for which a such limit
cycles exist. This completes the proof of Theorem \ref{main
theorem}.

Now we present an explicit example of a continuous piecewise linear
differential system \eqref{eq inicial} in $\R^4$, and repeating for
it the proof of Theorem \ref{main theorem} we will show it has one
limit cycle. Consider the following differential system
\begin{equation}\label{example}
\dot{x} = A_0 x  + \e (A x + \varphi(x_1) b),
\end{equation}
where
\begin{equation*}
A_0=\left(
\begin{array}{cccc}
 0 & -1 & 0 & 0 \\
 1 & 0 & 0 & 0 \\
 0 & 0 & 0 & -3 \\
 0 & 0 & 3 & 0
\end{array}
\right),  A=\left(
\begin{array}{cccc}
 2 & 1 & 0 & 0 \\
 0 & 2 & 0 & 0 \\
 0 & 0 & -2 & -1 \\
 0 & 0 & 0 & \dfrac{18 \pi -\sqrt{3}}{9 \pi }
\end{array}
\right), b=\left(
\begin{array}{c}
-\dfrac{24 \pi }{3 \sqrt{3}+2 \pi }\\
1\\
\dfrac{9(3-2\sqrt{3} \pi)}{2}\\
-1
\end{array}
\right).
\end{equation*}
Doing the change or variables $x_1=r \cos \T$, $x_2=r \sin \T$,
$x_3=r_1 \cos(3\T+\T_1)$, $x_4$ $=r_1 \sin(3\T+\T_1)$ and taking
$\T$ as the new independent variable we obtain the system
\begin{equation}\label{example polar}
\begin{array}{RL}
r'(\T)=&\frac{dr}{d\T}=\e H_1(\T,r,\T_1,r_1) + \CO(\e^2),\\
r_1'(\T)=&\frac{dr_1}{d\T}=\e H_2(\T,r,\T_1,r_1) + \CO(\e^2),\\
\T_1'(\T)=&\frac{d\T_1}{d\T}=\e H_3(\T,r,\T_1,r_1) + \CO(\e^2),
\\ \vspace{0.2cm}
H_1(\T,r,\T_1,r_1)=&2r +\varphi(r \cos \T)  \sin \T+\cos \T \bigg(r
\sin \T-\frac{24 \pi  \varphi(r \cos \T) }{3 \sqrt{3}+2 \pi
}\bigg),\\ \vspace{0.2cm}
H_2(\T,r,\T_1,r_1)=&-\frac{1}{18\pi}\bigg(18 \pi  \varphi(r \cos \T)
\sin (3 \T +\T_1)+9 \pi  r_1 \sin (2 (3 \T +\T_1))\\
&+\sqrt{3} r_1+81 \pi  (2 \sqrt{3} \pi -3) \varphi(r \cos \T)  \cos
(3 \T+\T_1)-\\
&(\sqrt{3}-36 \pi ) r_1 \cos (2 (3 \T +\T_1))\bigg),\\
\vspace{0.2cm} H_3(\T,r,\T_1,r_1)=&\sin^2(3 \T +\T_1)+2 \sin (2 (3
\T +\T_1))-\frac{\sin (2 (3 \T +\T_1))}{6 \sqrt{3} \pi }+3
\sin^2\T\\
&-\frac{72 \pi  \varphi(r \cos \T)  \sin \T}{3 \sqrt{3} r+2 \pi
r}-\frac{\varphi(r \cos \T)  \cos (3 \T +\T_1)}{r_1}-\frac{3
\varphi(r \cos \T)  \cos \T}{r}\\
&+\frac{9 \pi  \sqrt{3} \varphi(r \cos \T)  \sin (3 \T
+\T_1)}{r_1}-\frac{27 \varphi(r \cos \T)  \sin (3 \T +\T_1)}{2 r_1}.
\end{array}
\end{equation}

After some computations the average function $h=(h_1,h_2,h_3)$
defined in \eqref{averaged function theorem} is
\begin{equation*}
\begin{array}{RL}
h_1(r,\T_1,r_1)=&4 \pi  r-\frac{24 \pi}{3 \sqrt{3}+2 \pi} \bigg(\pi
r + \frac{2 \sqrt{r^2-1}}{r}-2 r \arctan(\sqrt{r^2-1})\bigg),\\
\vspace{0.2cm} h_2(r,\T_1,r_1)=&\frac{\sqrt{3}}{3} \sin
\T_1+\frac{3}{2} (2
\sqrt{3} \pi -3) \sqrt{3} \cos \T_1-\frac{\sqrt{3}}{9}r_1,\\
\vspace{0.2cm} h_3(r,\T_1,r_1)=&\frac{9 \sqrt{3} \sin \T_1}{2
r_1}-\frac{9 \pi \sin \T_1}{r_1}+\frac{\sqrt{3} \cos \T_1}{3
r_1}-\frac{3}{2}\bigg(\sqrt{3}+\frac{2 \pi }{3}\bigg)+4 \pi.
\end{array}
\end{equation*}
In order to solve the system $h_1=h_2=h_3=0$ we can use the same
reasoning applied in the proof of Theorem \ref{main theorem}
obtaining that $(r^*,\T_1^*,r_1^*)=(2,\pi/2,3)$ is a zero of the
average function. Moreover if $J=J(r,\T_1,r_1)$ is the Jacobian
matrix of $h$, then $\det J(2,\pi/2,3) \neq 0$ which implies that we
have a simple zero. By Theorem \ref{averaging theorem} system
\eqref{example polar} and consequently system \eqref{example} has
one limit cycle for $|\e|>0$ sufficiently small.

\section{Proof of Theorem \ref{main theoremd}}\label{s4}

This section is devoted to prove Theorem \ref{main theoremd}.
According to Theorem \ref{discontinuous} the same kind of arguments
used for proving Theorem \ref{main theorem} can be applied to the
discontinuous system \eqref{eq iniciald}, obtaining that the average
function $h=(h_1,h_2,\ldots,h_{2(n-1)},h_{2n-1}):D_m \to \R^{n-1}$
defined in \eqref{MRf1} is
\begin{equation}\label{ja}
\begin{array}{RL}
h_1=&(a_{11} + a_{22})\pi r + b_1 \widetilde{I}_1,\\
h_{2(j-1)}=&(a_{(2j-1)(2j-1)} + a_{(2j)(2j)})\pi r_{j-1} + (b_{2j-1}
\cos \T_{j-1} + b_{2j}\sin \T_{j-1}) \widetilde{I}_j,\\
h_{2j-1}=&(a_{(2j)(2j-1)} - a_{(2j-1)(2j)} +(2j-1)(a_{12}-a_{21}))\pi- \\
&\frac{(2j-1) b_2 r_{j-1}\widetilde{I}_1 -r (b_{2j}\cos \T_{j-1} -
b_{2j-1}\sin \T_{j-1}) \widetilde{I}_j}{r r_{j-1}},
\end{array}
\end{equation}
for $j=2,3,\ldots,n$, where
\begin{equation*}
\widetilde{I}_j=\begin{cases}
-\dfrac{4}{(2j-1)} & \text{if $j$ is even},\\
\dfrac{4}{(2j-1)} & \text{if $j$ is odd}.
\end{cases}
\end{equation*}

In fact if we define \begin{equation*}
\begin{array}{RL}
\widetilde{I}_j=&\int_0^{2\pi} \psi(r \cos \T) \cos((2j-1)\T) d\T,\\
\widetilde{J}_j=&\int_0^{2\pi} \psi(r \cos \T) \sin((2j-1)\T) d\T,
\end{array}
\end{equation*}
where $\psi$ is the piecewise linear function given by \eqref{eq phi
0}. Then we have that
\begin{equation*}
\begin{array}{RL}
\widetilde{I}_j=&\int_0^{2\pi} \psi(r \cos \T) \cos((2j-1)\T) d\T\\
=&\int_0^{\pi/2}  \cos((2j-1)\T) d\T-\int_{\pi/2}^{3\pi/2}
\cos((2j-1)\T) d\T + \int_{3\pi/2}^{2\pi}  \cos((2j-1)\T) d\T\\
=&-\frac{4}{(2j-1)}\cos(j \pi),
\end{array}
\end{equation*} and \begin{equation*}
\begin{array}{RL}
\widetilde{J}_j=&\int_0^{2\pi} \psi(r \cos \T) \sin((2j-1)\T) d\T\\
=&\int_0^{\pi/2}  \sin((2j-1)\T) d\T-\int_{\pi/2}^{3\pi/2}
\sin((2j-1)\T) d\T + \int_{3\pi/2}^{2\pi}  \sin((2j-1)\T) d\T\\
=&-\frac{4}{(2j-1)}\sin(2j\pi)\cos(j\pi)=0.
\end{array}
\end{equation*}

Note that $\widetilde{I}_j$ is a constant real number different from
zero, and hence $h_1$ is a straight line, and then system \eqref{ja}
has at most one positive zero. Moreover if we choose conveniently
the coefficients $b_1$ $a_{11}$ and $a_{22}$ we can find a simple
positive zero of system \eqref{ja}. This completes the proof of
Theorem \ref{main theoremd}.

\section{Proof of Theorems \ref{main theorem1} and \ref{main theoremd1}}\label{s5}

Doing the change of coordinates
\begin{equation*}
\begin{split}
&x_1=r\cos\T, \quad \quad x_2=r\sin\T,\\
&x_{2j-1}=r_{j-1}\cos(j\T+\T_{j-1}), \quad \quad x_{2j}=r_{j-1}
\sin(j\T+\T_{j-1}) \quad j \in \{2,3,\ldots,n\},
\end{split}
\end{equation*}
for $j=2,3,\ldots,n$, to the continuous piecewise linear
differential system \eqref{eq inicial1}, and working as in the proof
of Theorem \ref{main theorem} we obtain that the average function
$h=(h_1,h_2,\ldots$, $h_{2n-1})$  defined in \eqref{averaged
function theorem} now is given by
\begin{equation}
\label{arbitrary system 2}
\begin{array}{RL}
h_1=&(a_{11} + a_{22})\pi r + b_1 I_1(r),\\
h_{2(j-1)}=&(a_{(2j-1)(2j-1)} + a_{(2j)(2j)})\pi r_{j-1} + (b_{2j-1}
\cos \T_{j-1} + b_{2j}\sin \T_{j-1}) I_{j}(r),\\
h_{2j-1}=&(a_{(2j)(2j-1)} - a_{(2j-1)(2j)} +j(a_{12}-a_{21}))\pi- \\
&\frac{j b_2 r_{j-1}I_1(r) -r (b_{2j}\cos \T_{j-1} - b_{2j-1}\sin
\T_{j-1}) I_{j}(r)}{r r_{j-1}},
\end{array}
\end{equation}
where \begin{equation*}
\begin{array}{RL}
I_j(r)=&\int_0^{2\pi} \varphi(r \cos \T) \cos(j\T) d\T.\\
\end{array}
\end{equation*}
Using exactly the same arguments than in the proof of Lemma
\ref{lemma4} is possible to prove that
\begin{equation*}
\begin{array}{RL}
I_j(r)=&\begin{cases}
\pi r & \text{if $j=1$ and $0<r\leq 1$},\\
0 & \text{if $j$ is even and $0<r\leq 1$},\\
L_j(r)& \text{if $j$ is odd and $r>1$},
\end{cases}
\end{array}
\end{equation*}
where
\[
L_j(r)=\frac{4}{j(j^2-1)} \bigg(j \sqrt{r^2-1} \cos (j
\arctan(\sqrt{r^2-1}))- \sin (j \arctan(\sqrt{r^2-1}))\bigg).
\]

The simple zeros of system \eqref{arbitrary system 2} provide the
existence of limit cycles for system \eqref{eq inicial1} but since
$I_j(r)=0$ if $j$ is even and $r>1$, the variables $\T_{j-1}$, for
$j=2,4,6,...$ do not appear in the system
$h_1=h_2=\ldots=h_{2n-1}=0$, so either this system has no zeros, or
if it has zeros, then it has a continuum of zeros, and therefore the
assumption (ii) of the averaging Theorem \ref{eq inicial averaging}
does not hold, and consequently the averaging theory cannot say
anything about the limit cycles of system \eqref{eq inicial1}. The
same occurs for the case $0<r \leq 1$. So we conclude that, using
the averaging theory of first order, we can say nothing about the
number of the limit cycles of system \eqref{eq inicial1}. This
completes the proof of Theorems \ref{main theorem1}.

Now if we consider the discontinuous piecewise linear differential
system \eqref{eq inicial1d}, then its average function
$h=(h_1,h_2,\ldots,h_{2n-1})$ defined in \eqref{MRf1} is
\begin{equation}
\label{arbitrary system 3}
\begin{array}{RL}
h_1=&(a_{11} + a_{22})\pi r + b_1 \widetilde{I}_1,\\
h_{2(j-1)}=&(a_{(2j-1)(2j-1)} + a_{(2j)(2j)})\pi r_{j-1} + (b_{2j-1}
\cos \T_{j-1} + b_{2j}\sin \T_{j-1}) \widetilde{I}_{j},\\
h_{2j-1}=&(a_{(2j)(2j-1)} - a_{(2j-1)(2j)} +j(a_{12}-a_{21}))\pi- \\
&\frac{j b_2 r_{j-1}\widetilde{I}_1 -r (b_{2j}\cos \T_{j-1} -
b_{2j-1}\sin \T_{j-1}) \widetilde{I}_{j}}{r r_{j-1}},
\end{array}
\end{equation} where \begin{equation*}
\begin{array}{RL}
\widetilde{I}_j=&\int_0^{2\pi} \psi(r \cos \T) \cos(j\T) d\T.\\
\end{array}
\end{equation*}
Again we have that
\begin{equation*}
\widetilde{I}_j=\int_0^{2\pi} \psi(r \cos \T) \cos(j\T) d\T
=\begin{cases}
0 & \text{if $j$ is even},\\
\pm \dfrac{4}{(2j-1)} & \text{if $j$ is odd},
\end{cases}
\end{equation*}
and we can see that again either no zeros of the function $h$, or a
continuum of zeros, concluding that the averaging theory of first
order given by Theorem \ref{discontinuous} does not say anything
about the limit cycles of system \eqref{eq inicial1d}. This
completes the proof of Theorems \ref{main theoremd1}.

\section*{Acknowledgements}

The first authors is partially supported by the MINECO grants MTM2013-40998-P and MTM2016-77278-P (FEDER) and the AGAUR grant 2014 SGR568. 
The second author is supported by a Projeto Tem\'{a}tico FAPESP number
2014/00304-2. The third author has a PhD fellowship from CNPq-Brazil.



\end{document}